\documentclass[12pt,a4paper]{article}
\usepackage{amssymb}

\usepackage{amsthm}

\newtheorem{definition}{Definition}[section]
\newtheorem{proposition}{Proposition}[section]
\newtheorem{theorem}{Theorem}[section]
\newtheorem{remark}{Remark}[section]
\newtheorem{corollary}{Corollary}[section]
\newtheorem{lemma}{Lemma}[section]

\begin{document}

\author{Nicoleta Aldea and Gheorghe Munteanu}
\title{On projective invariants of the complex Finsler spaces}
\date{}
\maketitle

\begin{abstract}
In this paper the projective curvature invariants of a complex Finsler space
are obtained. By means of these invariants the notion of complex Douglas
space is then defined. A special approach is devoted to obtain the
equivalence conditions that a complex Finsler space should be Douglas. It is
shown that any weakly K\"{a}hler Douglas space is a complex Berwald space. A
projective curvature invariant of Weyl type characterizes the complex
Berwald spaces. They must be either purely Hermitian of constant holomorphic
curvature or non purely Hermitian of vanish holomorphic curvature. The
locally projectively flat complex Finsler metrics are also studied.
\end{abstract}

\begin{flushleft}
\strut \textbf{2010 Mathematics Subject Classification:} 53B40, 53C60.

\textbf{Key words and phrases:} projectively related complex Finsler
metrics, projective curvature invariants, complex Douglas space, locally
projectively flat complex Finsler metric.
\end{flushleft}

\section{Introduction}

\setcounter{equation}{0}The study of projective real Finsler spaces was
initiated by L. Berwald, \cite{Bw1,Bw2}, and his studies mainly concern the
two dimensional Finsler spaces. Further substantial contributions on this
topic came later from Rapcs\'{a}k \cite{Ra}, Misra \cite{Mi} and,
especially, from Z. Szabo \cite{Sz1} and M. Matsumoto \cite{Map}. The
problem of projective Finsler spaces is strongly connected to projectively
related sprays, as Z. Shen pointed out in \cite{Sh1}. The topic of
projective real Finsler spaces continues to be of interest for some
projective invariants: Douglas curvature, Weyl curvature and others. The
exploration of these projective invariants leads to the special classes of
metrics such as the Douglas metrics and the Finsler metrics of scalar flag
curvature, (\cite{B-M1,B-M2,C-ZS,Li,B-C}, etc.).

Few general themes from projective real Finsler geometry are broached in
complex Finsler geometry, \cite{Al-Mu}. Two complex Finsler metrics $L$ and $%
\tilde{L}$, on a common underlying manifold $M$, are called projectively
related if any complex geodesic curve, in \cite{A-P}' s sense, of the first
is also complex geodesic curve for the second and the other way around. This
means that between the spray coefficients $G^i$ and $\tilde{G}^i$ there is a
so called projective change $\tilde{G}^i=G^i+B^i+P\eta ^i,$ where $P$ is a
smooth function on $T^{\prime }M$ with complex values and $B^i:=\frac 12(%
\tilde{\theta}^{*i}-\theta ^{*i}).$ Although the Chern-Finsler complex
nonlinear connection, with the local coefficients $N_j^i,$ is the main tool
in complex Finsler geometry (\cite{A-P, Mub}), in this study we use the
canonical complex nonlinear connection because it derives from a complex
spray, i.e. $\stackrel{c}{N_j^i}:=\dot{\partial}_jG^i$ and $G^i=\frac
12N_j^i\eta ^j.$

Using some ideas from the real case, our aim in the present paper is to
study the above mentioned projective change. It gives rise to projective
curvature invariants of Douglas and Weyl types. Associated to the canonical
complex nonlinear connection we have the complex linear connection of
Berwald type which is an important tool in our approach.

Subsequently, we have made an overview of the paper's content.

In \S 2, some preliminary properties of the $n$ - dimensional complex
Finsler spaces are stated. We prove that the complex Finsler spaces which
are weakly K\"{a}hler and generalized Berwald are complex Berwald spaces
(Theorem 2.1).

In \S 3, the structure equations satisfied by the connection form of the
complex linear connection of Berwald type are emphasized. Next, we derive
some of Bianchi identities which specify the relations among the covariant
derivatives of the curvature coefficients of this complex linear connection.

A first class of projective curvature invariants obtained by successive
vertical differentiations of the projective change is explored in \S 4. We
find three projective invariants of Douglas type and by means of them are
defined the complex Douglas spaces. The necessary and sufficient conditions
in which a complex Finsler space is Douglas are contained in Theorem 4.2.

The study of the weakly K\"{a}hler projective changes is more significant.
We prove that the weakly K\"{a}hler Douglas spaces are complex Berwald
spaces, (Theorem 5.2). A projective curvature invariant of Weyl type $%
W_{jkh}^{i}$, which has the same formal form as in the real case, is
obtained. It is vanishing in the K\"{a}hler context. For the complex Berwald
spaces another projective curvature invariant of Weyl type $W_{j\bar{k}%
h}^{i} $ is found. We show that $W_{j\bar{k}h}^{i}=0$ if and only if the
space is either purely Hermitian with the holomorphic curvature $\mathcal{K}%
_{F}$ equal to a constant value or non purely Hermitian with $\mathcal{K}%
_{F}=0$, (Theorem 5.4).

The last part of the paper, \S 6, is devoted to the locally projectively
flat complex Finsler metrics. The necessary and sufficient conditions for
the locally projectively flat complex Finsler metrics and other
characterizations are established in Theorems 6.2, 6.3 and Proposition 6.2.
Finally, the locally projectively flat complex Finsler metrics are
exemplified, better illustrating the interest for this work, (Theorem 6.4).

\section{Preliminaries}

\setlength\arraycolsep{3pt} \setcounter{equation}{0}Let $M$ be a
n-dimensional complex manifold and $z=(z^k)_{k=\overline{1,n}}$ be the
complex coordinates in a local chart. The complexified of the real tangent
bundle $T_CM$ splits into the sum of holomorphic tangent bundle $T^{\prime
}M $ and its conjugate $T^{\prime \prime }M$. The bundle $T^{\prime }M$ is
itself a complex manifold and the local coordinates in a local chart will be
denoted by $u=(z^k,\eta ^k)_{k=\overline{1,n}}.$ These are changed into $%
(z^{\prime k},\eta ^{\prime k})_{k=\overline{1,n}}$ by the rules $z^{\prime
k}=z^{\prime k}(z)$ and $\eta ^{\prime k}=\frac{\partial z^{\prime k}}{%
\partial z^l}\eta ^l.$

A \textit{complex Finsler space} is a pair $(M,F)$, where $F:T^{\prime
}M\rightarrow \mathbb{R}^{+}$ is a continuous function satisfying the
conditions:

\textit{i)} $L:=F^2$ is smooth on $\widetilde{T^{\prime }M}:=T^{\prime
}M\backslash \{0\};$

\textit{ii)} $F(z,\eta )\geq 0$, the equality holds if and only if $\eta =0;$

\textit{iii)} $F(z,\lambda \eta )=|\lambda |F(z,\eta )$ for $\forall \lambda
\in \mathbb{C}$;

\textit{iv)} the Hermitian matrix $\left( g_{i\bar{j}}(z,\eta )\right) $ is
positive definite, where $g_{i\bar{j}}:=\frac{\partial ^2L}{\partial \eta
^i\partial \bar{\eta}^j}$ is the fundamental metric tensor. Equivalently, it
means that the indicatrix is strongly pseudo-convex.

Consequently, from $iii$) we have $\frac{\partial L}{\partial \eta ^k}\eta
^k=\frac{\partial L}{\partial \bar{\eta}^k}\bar{\eta}^k=L,$ $\frac{\partial
g_{i\bar{j}}}{\partial \eta ^k}\eta ^k=\frac{\partial g_{i\bar{j}}}{\partial
\bar{\eta}^k}\bar{\eta}^k=0$ and $L=g_{i\bar{j}}\eta ^i\bar{\eta}^j.$

Roughly speaking, the geometry of a complex Finsler space consists of the
study of the geometric objects of the complex manifold $T^{\prime }M$
endowed with the Hermitian metric structure defined by $g_{i\bar{j}}.$

Therefore, the first step is to study sections of the complexified tangent
bundle of $T^{\prime }M,$ which is decomposed in the sum $T_C(T^{\prime
}M)=T^{\prime }(T^{\prime }M)\oplus T^{\prime \prime }(T^{\prime }M)$. Let $%
VT^{\prime }M\subset T^{\prime }(T^{\prime }M)$ be the vertical bundle,
locally spanned by $\{\frac \partial {\partial \eta ^k}\}$, and $VT^{\prime
\prime }M$ be its conjugate.

At this point, the idea of complex nonlinear connection, briefly $(c.n.c.),$
is an instrument in 'linearization' of this geometry. A $(c.n.c.)$ is a
supplementary complex subbundle to $VT^{\prime }M$ in $T^{\prime }(T^{\prime
}M)$, i.e. $T^{\prime }(T^{\prime }M)=HT^{\prime }M\oplus VT^{\prime }M.$
The horizontal distribution $H_uT^{\prime }M$ is locally spanned by $\{\frac
\delta {\delta z^k}=\frac \partial {\partial z^k}-N_k^j\frac \partial
{\partial \eta ^j}\}$, where $N_k^j(z,\eta )$ are the coefficients of the $%
(c.n.c.)$, i.e. they transform by a certain rule
\begin{equation}
N_j^{\prime i}\frac{\partial z^{\prime j}}{\partial z^k}=\frac{\partial
z^{\prime i}}{\partial z^j}N_k^j-\frac{\partial ^2z^{\prime i}}{\partial
z^j\partial z^k}\eta ^j.  \label{III.1.10}
\end{equation}
The pair $\{\delta _k:=\frac \delta {\delta z^k},\dot{\partial}_k:=\frac
\partial {\partial \eta ^k}\}$ will be called the adapted frame of the $%
(c.n.c.)$ which obey to the change rules $\delta _k=\frac{\partial z^{\prime
j}}{\partial z^k}\delta _j^{\prime }$ and $\dot{\partial}_k=\frac{\partial
z^{\prime j}}{\partial z^k}\dot{\partial}_j^{\prime }.$ By conjugation
everywhere we have obtained an adapted frame $\{\delta _{\bar{k}},\dot{%
\partial}_{\bar{k}}\}$ on $T_u^{\prime \prime }(T^{\prime }M).$ The dual
adapted bases are $\{dz^k,\delta \eta ^k\}$ and $\{d\bar{z}^k,\delta \bar{%
\eta}^k\}.$

Let us consider $T$ the natural tangent structure which behaves on $%
T^{\prime }(T^{\prime }M)$ by $T(\frac \partial {\partial z^k})=\frac
\partial {\partial \eta ^k}$ ; $T(\frac \partial {\partial \eta ^k})=0,$ and
it is globally defined, (see \cite{Mub}).

\begin{definition}
\cite{Mub}. A vector field $S\in T^{\prime }(T^{\prime }M)$ is a complex
spray if $T\circ S=\Gamma ,$ where $\Gamma =\eta ^k\frac \partial {\partial
z^k}$ is the complex Liouville vector field.
\end{definition}

Locally, this condition of complex spray can be expressed as follows
\begin{equation}
S=\eta ^k\frac \partial {\partial z^k}-2G^k(z,\eta )\frac \partial {\partial
\eta ^k}\;.  \label{III.1.20}
\end{equation}

Under the changes of complex coordinates on $T^{\prime }M,$ the coefficients
$G^k$ of the spray $S$ are transformed by the rule
\begin{equation}
2G^{\prime i}=2G^k\frac{\partial z^{\prime i}}{\partial z^k}-\frac{\partial
^2z^{\prime i}}{\partial z^j\partial z^k}\eta ^j\eta ^k.  \label{III.1.21}
\end{equation}

Between the notions of complex spray and $(c.n.c.)$ there exists an
interdependence, one determining the other. Differentiating (\ref{III.1.21})
with respect to $\eta ^j$ it follows that the functions $N_j^i:=\frac{%
\partial G^i}{\partial \eta ^j}$ satisfy the rule (\ref{III.1.10}), and
hence $N_j^i$ define a nonlinear connection. Conversely, any $(c.n.c.)$
determines a complex spray. Indeed, a simple computation shows that if $%
N_j^i $ are the coefficients of a $(c.n.c.)$ then $\frac 12N_j^i\;\eta ^j$
satisfy (\ref{III.1.21}) and hence, they define a complex spray. \

A $(c.n.c.)$ related only to the fundamental function of the complex Finsler
space $(M,F)$ is the so called Chern-Finsler $(c.n.c.),$ (cf. \cite{A-P}),
with the local coefficients $N_{j}^{i}:=g^{\overline{m}i}\frac{\partial g_{l%
\overline{m}}}{\partial z^{j}}\eta ^{l}.$ Further on $\delta _{k}$ is the
adapted frame of the Chern-Finsler $(c.n.c.).$ A Hermitian connection $D$,
of $(1,0)-$ type, which satisfies in addition $D_{JX}Y=JD_{X}Y,$ for all $X$
horizontal vectors and $J$ the natural complex structure of the manifold, is
the Chern-Finsler connection (\cite{A-P})$.$ It is locally given by the
following coefficients (cf. \cite{Mub}):
\begin{equation}
L_{jk}^{i}:=g^{\overline{l}i}\delta _{k}g_{j\overline{l}}=\dot{\partial}%
_{j}N_{k}^{i}\;\;;\;C_{jk}^{i}:=g^{\overline{l}i}\dot{\partial}_{k}g_{j%
\overline{l}}.  \label{1.3}
\end{equation}

Recall that $R_{j\overline{h}k}^i:=-\delta _{\overline{h}}L_{jk}^i-(\delta _{%
\overline{h}}N_k^l)C_{jl}^i$ are $h\bar{h}$ - curvatures coefficients of
Chern-Finsler connection. According to \cite{A-P}, p. 108, \cite{Mub}, p.
81, the \textit{holomorphic curvature} of the complex Finsler space $(M,F)$
in direction $\eta $ is
\begin{equation}
\mathcal{K}_F(z,\eta )=\frac 2{L^2}R_{\bar{r}j\bar{k}h}\bar{\eta}^r\eta ^j%
\bar{\eta}^k\eta ^h,  \label{1.8}
\end{equation}
where $R_{\bar{r}j\bar{k}h}:=R_{j\bar{k}h}^ig_{i\bar{r}}.$

In \cite{A-P}'s terminology, the complex Finsler space $(M,F)$ is \textit{%
strongly K\"{a}hler} iff $T_{jk}^i=0,$ \textit{K\"{a}hler}$\;$iff $%
T_{jk}^i\eta ^j=0$ and \textit{weakly K\"{a}hler }iff\textit{\ } $g_{i%
\overline{l}}T_{jk}^i\eta ^j\overline{\eta }^l=0,$ where $%
T_{jk}^i:=L_{jk}^i-L_{kj}^i.$ In \cite{C-S} it is proved that strongly
K\"{a}hler and K\"{a}hler notions actually coincide. We notice that in the
particular case of the complex Finsler metrics which come from Hermitian
metrics on $M,$ so-called \textit{purely Hermitian metrics} in \cite{Mub},
(i.e. $g_{i\overline{j}}=g_{i\overline{j}}(z)$)$,$ all those nuances of
K\"{a}hler are same. On the other hand, as in Aikou's work \cite{Ai}, a
complex Finsler space which is K\"{a}hler and $L_{jk}^i=L_{jk}^i(z)$ is
named \textit{complex Berwald} space.

In \cite{Mub} it is proved that the Chern-Finsler $(c.n.c.)$ does not
generally come from a complex spray except when the complex metric is weakly
K\"{a}hler. But, its local coefficients $N_j^i$ always determine a complex
spray with coefficients $G^i=\frac 12N_j^i\eta ^j.$ Further, $G^i$ induce a $%
(c.n.c.)$ denoted by $\stackrel{c}{N_j^i}:=\dot{\partial}_jG^i$ and called
\textit{canonical} in \cite{Mub}, where it is proved that it coincides with
Chern-Finsler $(c.n.c.)$ if and only if the complex Finsler metric is
K\"{a}hler. With respect to the canonical $(c.n.c.),$ we consider the frame $%
\{\stackrel{c}{\delta _k},\dot{\partial}_k\},$ where $\stackrel{c}{\delta _k}%
:=\frac \partial {\partial z^k}-\stackrel{c}{N_k^j}\dot{\partial}_j,$ and
its dual coframe $\{dz^k,\stackrel{c}{\delta }\eta ^k\},$ where $\stackrel{c%
}{\delta }\eta ^k:=d\eta ^k+\stackrel{c}{N_j^k}dz^j.$ Moreover, we associate
to the canonical $(c.n.c.)$ a complex linear connection of Berwald type $%
B\Gamma $ with its connection form
\begin{equation}
\omega _j^i(z,\eta )=G_{jk}^idz^k+G_{j\bar{k}}^id\bar{z}^k,  \label{1.4}
\end{equation}
where $G_{jk}^i:=\dot{\partial}_k\stackrel{c}{N_j^i}$ and $G_{j\bar{k}}^i:=%
\dot{\partial}_{\bar{k}}\stackrel{c}{N_j^i}.$ Note that the spray
coefficients perform $2G^i=N_j^i\eta ^j=\stackrel{c}{N_j^i}\eta
^j=G_{jk}^i\eta ^j\eta ^k=L_{jk}^i\eta ^j\eta ^k$.

An extension of the complex Berwald spaces, directly related to the $B\Gamma
$ connection, is called by us \textit{generalized Berwald} in \cite{Al-Mu}.
It is with the coefficients $G_{jk}^i$ depending only on the position $z,$
equivalently with either $\dot{\partial}_{\bar{h}}G^i=0$ or $B\Gamma $ is of
$(1,0)$ - type. Since in the K\"{a}hler case $G_{jk}^i=L_{jk}^i,$ any
complex Berwald space is generalized Berwald.

In Abate-Patrizio's sense, (\cite{A-P} p. 101), a complex geodesic curve is
given by $D_{T^h+\overline{T^h}}T^h=\theta ^{*}(T^h,\overline{T^h}),$ where $%
\theta ^{*}=g^{\bar{m}k}g_{i\bar{p}}(L_{\bar{j}\bar{m}}^{\bar{p}}-L_{\bar{m}%
\bar{j}}^{\bar{p}})dz^i\wedge d\bar{z}^j\otimes \delta _k,$ for which it is
proven in \cite{Mub} that $\theta ^{*k}=2g^{\bar{j}k}\stackrel{c}{\delta _{%
\bar{j}}}L$ and $\theta ^{*i}$ is vanishing if and only if the space is
weakly K\"{a}hler. Thus, the equations of a complex geodesic $z=z(s)$ of $%
(M,F)$, with $s$ a real parameter, in \cite{A-P}' s sense can be rewritten
as
\begin{equation}
\frac{d^2z^i}{ds^2}+2G^i(z(s),\frac{dz}{ds})=\theta ^{*i}(z(s),\frac{dz}{ds}%
)\;;\;i=\overline{1,n},  \label{40}
\end{equation}
where by $z^i(s),$ $i=\overline{1,n},$ we denote the coordinates along of
curve $z=z(s).$

Note that the functions $\theta ^{*i}$ are $(1,1)$ - homogeneous with
respect to $\eta ,$ i.e. $(\dot{\partial}_k\theta ^{*i})\eta ^k=\theta ^{*i}$
and $(\dot{\partial}_{\bar{k}}\theta ^{*i})\bar{\eta}^k=\theta ^{*i}.$

Next, we emphasize some properties of the complex Finsler spaces.

\begin{lemma}
Let $(M,F)$ be a complex Finsler space. Then, $(\dot{\partial}_{\bar{k}%
}G^i)\eta _i=0,$ where $\eta _i:=\dot{\partial}_iL.$
\end{lemma}

\begin{proof}
It results differentiating $G^ig_{i\bar{j}}=\frac 12\frac{%
\partial g_{h\bar{j}}}{\partial z^s}\eta ^h\eta ^s$ with respect to $\bar{%
\eta}^k$ and then contracting on it by $\bar{\eta}^j.$
\end{proof}

Also, it is necessary to compute

$\dot{\partial}_k\theta ^{*i}=2\dot{\partial}_k(g^{\bar{j}i}\stackrel{c}{%
\delta _{\bar{j}}}L)=-2g^{\bar{j}l}g^{\bar{m}i}(\dot{\partial}_kg_{l\bar{m}%
})(\stackrel{c}{\delta _{\bar{j}}}L)+2g^{\bar{j}i}\dot{\partial}_k(\stackrel{%
c}{\delta _{\bar{j}}}L)$

$=-\theta ^{*l}C_{kl}^i+2g^{\bar{j}i}\dot{\partial}_k[\frac{\partial L}{%
\partial \bar{z}^j}-\stackrel{c}{N_{\bar{j}}^{\bar{r}}}(\dot{\partial}_{\bar{%
r}}L)]$

$=-\theta ^{*l}C_{kl}^i+2g^{\bar{j}i}[\frac{\partial ^2L}{\partial \eta
^k\partial \bar{z}^j}-(\dot{\partial}_k\stackrel{c}{N_{\bar{j}}^{\bar{r}}})(%
\dot{\partial}_{\bar{r}}L)-\stackrel{c}{N_{\bar{j}}^{\bar{r}}}g_{k\bar{r}}].$

Now, using Lemma 2.1 and $\frac{\partial ^2L}{\partial \eta ^k\partial \bar{z%
}^j}=N_{\bar{j}}^{\bar{r}}g_{k\bar{r}}$ we obtain

$(\dot{\partial}_k\stackrel{c}{N_{\bar{j}}^{\bar{r}}})(\dot{\partial}_{\bar{r%
}}L)=(\dot{\partial}_k\stackrel{c}{N_{\bar{j}}^{\bar{r}}})\bar{\eta}_r=[\dot{%
\partial}_{\bar{j}}(\dot{\partial}_kG^{\bar{r}})]\bar{\eta}_r=\dot{\partial}%
_{\bar{j}}[(\dot{\partial}_kG^{\bar{r}})\bar{\eta}_r]-(\dot{\partial}_kG^{%
\bar{r}})(\dot{\partial}_{\bar{j}}\bar{\eta}_r)$

$=-(\dot{\partial}_kG^{\bar{r}})C_{l\bar{r}\bar{j}}\eta ^l,$ where $\bar{\eta%
}_r:=\dot{\partial}_{\bar{r}}L$ and $C_{l\bar{r}\bar{j}}\eta ^l:=\dot{%
\partial}_{\bar{j}}\bar{\eta}_r.$

Therefore,
\begin{equation}
\dot{\partial}_k\theta ^{*i}=-\theta ^{*l}C_{kl}^i+2g^{\bar{j}i}[(N_{\bar{j}%
}^{\bar{r}}-\stackrel{c}{N_{\bar{j}}^{\bar{r}}})g_{k\bar{r}}+(\dot{\partial}%
_kG^{\bar{r}})C_{l\bar{r}\bar{j}}\eta ^l].  \label{122}
\end{equation}

\begin{theorem}
Let $(M,F)$ be a complex Finsler space which is weakly K\"{a}hler and
generalized Berwald. Then it is a complex Berwald space.
\end{theorem}

\begin{proof}
Under given assumptions, the relation (\ref{122}) is $2g^{%
\bar{j}i}(N_{\bar{j}}^{\bar{r}}-\stackrel{c}{N_{\bar{j}}^{\bar{r}}})g_{k\bar{%
r}}=0,$ which contracted by $\frac 12g_{i\bar{m}}g^{\bar{s}k}$ gives $N_{%
\bar{m}}^{\bar{s}}-\stackrel{c}{N_{\bar{m}}^{\bar{s}}}=0,$ i.e. $F$ is
K\"{a}hler. This, together with the statement of generalized Berwald, proves
our claim.
\end{proof}

\section{Curvature forms and Bianchi identities}

\setlength\arraycolsep{3pt} \setcounter{equation}{0}We shall use the complex
linear connection of Berwald type $B\Gamma $ as our main tool to study the
projective geometry of the complex Finsler manifolds. The connection form of
$B\Gamma $ satisfy the following structure equations
\begin{equation}
d(dz^i)-dz^k\wedge \omega _k^i=h\Omega ^i\;;\;\;d(\stackrel{c}{\delta }\eta
^i)-\stackrel{c}{\delta }\eta ^k\wedge \omega _k^i=v\Omega ^i\;;\;\;d\omega
_j^i-\omega _j^k\wedge \omega _k^i=\Omega _j^i,  \label{II.7}
\end{equation}
and their conjugates, where $d$ is exterior differential with respect to the
canonical $(c.n.c.).$

Since
\begin{eqnarray*}
d(\stackrel{c}{\delta }\eta ^i) &=&d\stackrel{c}{N_j^i}\wedge dz^j=\frac
12K_{jk}^idz^k\wedge dz^j+\Theta _{j\bar{k}}^id\bar{z}^k\wedge dz^j \\
&&+G_{jk}^i\stackrel{c}{\delta }\eta ^k\wedge dz^j+G_{j\bar{k}}^i\stackrel{c%
}{\delta }\bar{\eta}^k\wedge dz^j
\end{eqnarray*}
and $G_{jk}^i=G_{kj}^i,$ the torsion and curvature forms are
\begin{eqnarray*}
h\Omega ^i &=&-G_{j\bar{k}}^idz^j\wedge d\bar{z}^k\;;\;\; \\
v\Omega ^i &=&-\frac 12K_{jk}^idz^j\wedge dz^k-\Theta _{j\bar{k}%
}^idz^j\wedge d\bar{z}^k-G_{j\bar{k}}^idz^j\wedge \stackrel{c}{\delta }\bar{%
\eta}^k-G_{j\bar{k}}^i\stackrel{c}{\delta }\eta ^j\wedge d\bar{z}^k\;; \\
\Omega _j^i &=&-\frac 12K_{jkh}^idz^k\wedge dz^h-\frac 12K_{j\bar{k}\bar{h}%
}^id\bar{z}^k\wedge d\overline{z}^h+K_{j\overline{h}k}^idz^k\wedge d%
\overline{z}^h \\
&&-G_{jkh}^idz^k\wedge \stackrel{c}{\delta }\eta ^h-G_{j\bar{k}\bar{h}}^id%
\bar{z}^k\wedge \stackrel{c}{\delta }\bar{\eta}^h-G_{j\overline{h}%
k}^idz^k\wedge \stackrel{c}{\delta }\overline{\eta }^h+G_{j\overline{h}k}^i%
\stackrel{c}{\delta }\eta ^k\wedge d\overline{z}^h,
\end{eqnarray*}
where

$K_{jk}^i:=\stackrel{c}{\delta _k}\stackrel{c}{N_j^i}-\stackrel{c}{\delta _j}%
\stackrel{c}{N_k^i}\;;\;\Theta _{j\bar{k}}^i:=\stackrel{c}{\delta _{\bar{k}}}%
\stackrel{c}{N_j^i};$ and

$K_{jkh}^i:=\stackrel{c}{\delta _h}G_{jk}^i-\stackrel{c}{\delta _k}%
G_{jh}^i+G_{jk}^lG_{lh}^i-G_{jh}^lG_{lk}^i$ ;

$K_{j\bar{k}\bar{h}}^i:=\stackrel{c}{\delta _{\bar{h}}}G_{j\bar{k}}^i-%
\stackrel{c}{\delta _{\bar{k}}}G_{j\bar{h}}^i+G_{j\bar{k}}^lG_{l\bar{h}%
}^i-G_{j\bar{h}}^lG_{l\bar{k}}^i$ ;

$K_{j\bar{k}h}^i:=\stackrel{c}{\delta _h}G_{j\bar{k}}^i-\stackrel{c}{\delta
_{\bar{k}}}G_{jh}^i+G_{j\bar{k}}^lG_{lh}^i-G_{jh}^lG_{l\bar{k}}^i$ are $hh$-$%
,$ $\bar{h}\bar{h}$- and $h\bar{h}$- curvature tensors, respectively;

$G_{jkh}^i:=\dot{\partial}_hG_{jk}^i\;;\;G_{j\bar{k}\bar{h}}^i=\dot{\partial}%
_{\bar{h}}G_{j\bar{k}}^i\;;\;G_{j\bar{k}h}^i:=\dot{\partial}_hG_{j\bar{k}}^i$
are $hv$-, $\bar{h}\bar{v}$- and $h\bar{v}$- curvature tensors,
respectively. Moreover, they have properties
\begin{eqnarray*}
K_{jkh}^i &=&\dot{\partial}_jK_{kh}^i\;;\;K_{jkh}^i\eta ^j=K_{kh}^i\;;\;K_{j%
\bar{k}\bar{h}}^i+K_{j\bar{h}\bar{k}}^i=0; \\
G_{j\bar{k}h}^i\eta ^j &=&G_{h\bar{k}}^i\;;\;G_{j\bar{k}\bar{h}}^i\bar{\eta}%
^h=-G_{j\bar{k}}^i.
\end{eqnarray*}

Note that we preferred to denote by $K_{jkh}^{i}$ the horizontal curvature
tensors of $B\Gamma $, instead of classical real notation $R_{jkh}^{i}$. In
this way, we avoid any confusion with the horizontal curvatures coefficients
of the Chern-Finsler connection from (\ref{1.8}).

Taking the exterior differential of the third structure equation from (\ref
{II.7}), it results
\begin{equation}
-\Omega _{j}^{l}\wedge \omega _{l}^{i}+\omega _{j}^{l}\wedge \Omega
_{l}^{i}=d\Omega _{j}^{i},  \label{II.2.5'}
\end{equation}
which leads to sixteen Bianchi identities. We mention here only some of
these, which are needed for our proposed study
\begin{eqnarray*}
\dot{\partial}_{r}G_{jkh}^{i} &=&\dot{\partial}_{h}G_{jkr}^{i}\;;\;\dot{%
\partial}_{r}G_{jkh}^{i}=\dot{\partial}_{h}G_{jkr}^{i}\;;\;\dot{\partial}%
_{r}G_{j\overline{k}h}^{i}=\dot{\partial}_{h}G_{j\overline{k}r}^{i}; \\
\dot{\partial}_{\bar{r}}G_{jkh}^{i} &=&\dot{\partial}_{h}G_{j\overline{r}%
k}^{i}\;;\;\dot{\partial}_{r}G_{j\bar{k}\bar{h}}^{i}=\dot{\partial}_{\bar{h}%
}G_{j\overline{k}r}^{i}\;;\;\dot{\partial}_{\bar{r}}G_{j\bar{k}\bar{h}}^{i}=%
\dot{\partial}_{\bar{h}}G_{j\bar{k}\bar{r}}^{i}\;;\;\dot{\partial}_{\bar{r}%
}G_{j\overline{h}k}^{i}=\dot{\partial}_{\bar{h}}G_{j\overline{r}k}^{i}.
\end{eqnarray*}
In the generalized Berwald case the following identities are true:
\[
\dot{\partial}_{r}K_{jkh}^{i}=0\;;\;\dot{\partial}_{\bar{r}%
}K_{jkh}^{i}=0\;;\;\dot{\partial}_{r}K_{j\overline{k}h}^{i}=0\;;\;\dot{%
\partial}_{\bar{r}}K_{j\overline{k}h}^{i}=0
\]
and for complex Berwald spaces we get
\begin{equation}
K_{j\bar{r}k|\bar{h}}^{i}=K_{j\bar{h}k|\bar{r}}^{i}\;;\;K_{j\overline{r}%
k|h}^{i}=K_{j\overline{r}h|k}^{i},  \label{111}
\end{equation}
where we denoted by '$_{|k}$ ' the horizontal covariant derivative with
respect to Chern-Finsler connection.

\section{ Projective invariants of a complex Finsler space}

\setlength\arraycolsep{3pt} \setcounter{equation}{0}Let $\tilde{L}$ be
another complex Finsler metric on the underlying manifold $M.$ Corresponding
to the metric $\tilde{L}$, we have the spray coefficients $\tilde{G}^i$ and
the functions $\tilde{\theta}^{*i}$. The complex Finsler metrics $L$ and $%
\tilde{L}$ on the manifold $M$, are called \textit{projectively related} if
these have the same complex geodesics as point sets. This means that for any
complex geodesic $z=z(s)$ of $(M,L)$ there is a transformation of its
parameter $s$, $\tilde{s}=\tilde{s}(s),$ with $\frac{d\tilde{s}}{ds}>0,$
such that $z=z(\tilde{s}(s))$ is a geodesic of $(M,\tilde{L})$, and
conversely.

\begin{theorem}
\cite{Al-Mu}. Let $L$ and $\tilde{L}$ be complex Finsler metrics on the
manifold $M$. Then $L$ and $\tilde{L}$ are projectively related if and only
if there is a smooth function $P$ on $T^{\prime }M$ with complex values,
such that
\begin{equation}
\tilde{G}^i=G^i+B^i+P\eta ^i;\;i=\overline{1,n},  \label{12'}
\end{equation}
where $B^i:=\frac 12(\tilde{\theta}^{*i}-\theta ^{*i}).$
\end{theorem}

The relations (\ref{12'}) between the spray coefficients $\tilde{G}^{i}$ and
$G^{i}$ of the projectively related complex Finsler metrics $L$ and $\tilde{L%
}$ is called \textit{projective change. }An equivalent form of this, (see
Lemma 3.2, \cite{Al-Mu}) is
\begin{equation}
\tilde{G}^{i}=G^{i}+S\eta ^{i}\mbox{   and   }\tilde{\theta}^{\ast i}=\theta
^{\ast i}+Q\eta ^{i};\;i=\overline{1,n},  \label{12''}
\end{equation}
where $S:=(\dot{\partial}_{k}P)\eta ^{k}$ is $(1,0)$ - homogeneous, $Q:=-2(%
\dot{\partial}_{\bar{k}}P)\bar{\eta}^{k}$ is $(0,1)$ - homogeneous and $S-%
\frac{1}{2}Q=P$.

Differentiating (\ref{12''}) with respect to $\eta ^{j}$ leads to
\begin{equation}
\stackrel{c}{\tilde{N}_{j}^{i}}=\stackrel{c}{N_{j}^{i}}+S_{j}\eta
^{i}+S\delta _{j}^{i}\mbox{   and   }\tilde{\theta}_{j}^{\ast i}=\theta
_{j}^{\ast i}+Q_{j}\eta ^{i}+Q\delta _{j}^{i},  \label{14}
\end{equation}
where $S_{j}:=\dot{\partial}_{j}S,$ $Q_{j}:=\dot{\partial}_{j}Q$, $\tilde{%
\theta}_{j}^{\ast i}:=\dot{\partial}_{j}\tilde{\theta}^{\ast i}$ and $\theta
_{j}^{\ast i}:=\dot{\partial}_{j}\theta ^{\ast i}$. Thus, $S_{j}-\frac{1}{2}%
Q_{j}=P_{j},$ with $P_{j}:=\dot{\partial}_{j}P.$

Now, to eliminate $S$ and $Q$ from (\ref{14}), we make the sum by $i=j.$
Since $S_i\eta ^i=S$ and $Q_i\eta ^i=0$, (\ref{14}) gives
\begin{equation}
S=\frac 1{n+1}(\stackrel{c}{\tilde{N}_i^i}-\stackrel{c}{N_i^i})%
\mbox{   and
}Q=\frac 1n(\tilde{\theta}_i^{*i}-\theta _i^{*i}).  \label{15}
\end{equation}

So that, $P=\frac 1{n+1}(\stackrel{c}{\tilde{N}_i^i}-\stackrel{c}{N_i^i}%
)-\frac 1{2n}(\tilde{\theta}_i^{*i}-\theta _i^{*i}).$ Substituting this in (%
\ref{12'}), we find that the projective change is
\begin{equation}
\tilde{G}^i=G^i+\frac 12(\tilde{\theta}^{*i}-\theta ^{*i})+\frac 1{n+1}(%
\stackrel{c}{\tilde{N}_l^l}-\stackrel{c}{N_l^l})\eta ^i-\frac 1{2n}(\tilde{%
\theta}_l^{*l}-\theta _l^{*l})\eta ^i,\;i=\overline{1,n}.  \label{15'}
\end{equation}
From here it results
\begin{equation}
D^i:=G^i-\frac 1{n+1}\stackrel{c}{N_l^l}\eta ^i-\frac 12(\theta ^{*i}-\frac
1n\theta _l^{*l}\eta ^i),  \label{16}
\end{equation}
which are the components of a projective invariant, under the projective
change (\ref{12'}).

\begin{proposition}
Let $(M,F)$ be a complex Finsler space. Then, $D^i$ are the local
coefficients of a complex spray if and only if $F$ is weakly K\"{a}hler.
\end{proposition}

\begin{proof}
First, $D^i$ satisfy the rule (\ref{III.1.21}), forasmuch $%
\stackrel{c}{N_l^l}\eta ^i$, $\theta ^{*i}$ and $\theta _l^{*l}\eta ^i$ have
changes all like vectors. Second, $D^i$ are $(2,0)$ - homogeneous if and
only if $\theta ^{*i}=\frac 1n\theta _l^{*l}\eta ^i$. The last relation
contracted by $\eta _i$ gives $0=\theta ^{*i}\eta _i=\frac 1n\theta _l^{*l}L.
$ Hence, $\theta _l^{*l}=0$ and so $\theta ^{*i}=0.$
\end{proof}

Further on, the projective change (\ref{12'}) gives rise to various
projective invariants. Indeed, some successive differentiations of (\ref{16}%
) with respect to $\eta $ and $\bar{\eta}$ give three \textit{projective
curvature invariants of Douglas type}
\begin{eqnarray}
D_{jkh}^i &=&G_{jkh}^i-\frac 1{n+1}[(\dot{\partial}_hD_{jk})\eta
^i+\sum_{(j,k,h)}D_{jh}\delta _k^i]  \label{17} \\
&&-\frac 12\{\theta _{jkh}^{*i}-\frac 1n[(\dot{\partial}_h\theta
_{ljk}^{*l})\eta ^i+\sum_{(j,k,h)}\theta _{ljh}^{*l}\delta _k^i]\};
\nonumber \\
D_{j\bar{k}\bar{h}}^i &=&G_{j\bar{k}\bar{h}}^i-\frac 1{n+1}[(\dot{\partial}%
_jD_{\bar{k}\bar{h}})\eta ^i+D_{\bar{k}\bar{h}}\delta _j^i]  \nonumber \\
&&-\frac 12\{\theta _{j\bar{k}\bar{h}}^{*i}-\frac 1n[(\dot{\partial}_{\bar{h}%
}\theta _{l\bar{k}j}^{*l})\eta ^i+\theta _{l\bar{k}\bar{h}}^{*l}\delta
_j^i]\};  \nonumber \\
D_{j\bar{k}h}^i &=&G_{j\bar{k}h}^i-\frac 1{n+1}[(\dot{\partial}_hD_{\bar{k}%
j})\eta ^i+D_{\bar{k}j}\delta _h^i+D_{\bar{k}h}\delta _j^i]  \nonumber \\
&&-\frac 12\{\theta _{j\bar{k}h}^{*i}-\frac 1n[(\dot{\partial}_h\theta _{l%
\bar{k}j}^{*l})\eta ^i+\theta _{l\bar{k}j}^{*l}\delta _h^i+\theta _{l\bar{k}%
h}^{*l}\delta _j^i]\},  \nonumber
\end{eqnarray}
where $D_{kh}:=G_{ikh}^i$, $D_{\bar{k}\bar{h}}:=G_{i\bar{k}\bar{h}}^i$ and $%
D_{\bar{k}h}:=G_{i\bar{k}h}^i$ are respectively, $hv$-, $\bar{h}\bar{v}$-
and $h\bar{v}$- \textit{Ricci tensors} and $\theta _{jkh}^{*i}:=\dot{\partial%
}_h\theta _{jk}^{*i},$ $\theta _{jk}^{*i}:=\dot{\partial}_k\theta _j^{*i},$ $%
\theta _{j\bar{k}h}^{*i}:=\dot{\partial}_{\bar{k}}\theta _{jh}^{*i},$ $%
\theta _{j\bar{k}\bar{h}}^{*i}:=\dot{\partial}_{\bar{k}}\theta _{j\bar{h}%
}^{*i}$ and $\theta _{j\bar{h}}^{*i}:=\dot{\partial}_{\bar{h}}\theta _j^{*i}=%
\dot{\partial}_j\theta _{\bar{h}}^{*i}.$ In (\ref{17}), $\displaystyle %
\sum_{(j,k,h)}$ is the cyclic sum.

\begin{definition}
A complex Finsler space $(M,F)$ is called complex Douglas space if the
invariants (\ref{17}) are vanishing.
\end{definition}

\begin{remark}
If $F$ is generalized Berwald, i.e. $G_{jk}^i(z)$, and weakly K\"{a}hler
then $G_{jkh}^i=G_{j\bar{k}\bar{h}}^i=G_{j\bar{k}h}^i=0$ and $D_{kh}=D_{\bar{%
k}\bar{h}}=D_{\bar{k}h}=\theta ^{*i}=0,$ and so the projective curvature
invariants of Douglas type are vanishing. Moreover, taking into account
Theorem 2.1 it results that any complex Berwald space is a complex Douglas
space.
\end{remark}

Subsequently, the key of the proofs is the strong maximum principle which
gives the constancy of the holomorphic and $0$ - homogenous functions.

\begin{lemma}
If one of $hv$-, $\bar{h}\bar{v}$- or $h\bar{v}$- Ricci tensors is vanishing
then they are all vanishing.
\end{lemma}

\begin{proof}
Supposing $D_{kh}=0,$ it results $G_{ikh}^i=0,$ which is
equivalent with $\dot{\partial}_hG_{ik}^i=0.$ By conjugation, $\dot{%
\partial}_{\bar{h}}G_{\bar{\imath}\bar{k}}^{\bar{\imath}}=0,$ and so, $G_{%
\bar{\imath}\bar{k}}^{\bar{\imath}}$ are holomorphic in $\eta .$ But, $G_{%
\bar{\imath}\bar{k}}^{\bar{\imath}}$ are 0 - homogeneous with respect to $%
\eta $ and so they depend only on $z,$ ($G_{\bar{\imath}\bar{k}}^{\bar{\imath%
}}=G_{\bar{\imath}\bar{k}}^{\bar{\imath}}(z)$). Hence, $G_{ik}^i$ depend
only on $z$ and $\dot{\partial}_{\bar{h}}G_{ik}^i=D_{\bar{h}k}=0$ which
contracted by $\eta ^k$ give $\dot{\partial}_{\bar{h}}\stackrel{c}{N_i^i}=0,$
i.e. $G_{i\bar{h}}^i=0.$ So that, $D_{\bar{k}\bar{h}}=\dot{\partial}_{\bar{h}%
}G_{i\bar{k}}^i=0.$

If $D_{\bar{k}\bar{h}}=0$ then $\dot{\partial}_{\bar{h}}G_{i\bar{k}}^i=0$
which contracted by $\bar{\eta}^h$ yield $G_{i\bar{k}}^i=0,$ because $G_{i%
\bar{k}\bar{h}}^i\bar{\eta}^h=-G_{i\bar{k}}^i.$ It results $\dot{\partial}_{%
\bar{k}}G_{ih}^i=0,$ i.e. $D_{\bar{k}h}=0.$ Further on using the
holomorphicity in $\eta $ and 0 - homogeneity of the coefficients $G_{ih}^i$
it results that $G_{ih}^i$ depend on $z$ alone$.$ So, $\dot{\partial}%
_jG_{ih}^i=0$ which give $D_{hj}=0.$

If $D_{\bar{k}h}=0$ then $\dot{\partial}_{\bar{k}}G_{ih}^i=0$ and similarly it
results that $G_{ih}^i$ depend only on $z$ and $G_{i\bar{k}}^i=0.$ This
implies $D_{hj}=D_{\bar{k}\bar{h}}=0.$
\end{proof}

Since $\theta ^{*i}$ are $(1,1)$ - homogeneous with respect to $\eta ,$ $%
\theta _k^{*i}\eta ^k=\theta ^{*i}$ and $\theta _{\bar{k}}^{*i}\bar{\eta}%
^k=\theta ^{*i}$ and so,
\begin{eqnarray}
\theta _{kj}^{*i}\eta ^k &=&0\;\;;\;\;\theta _{k\bar{h}}^{*i}\eta ^k=\theta
_{\bar{h}}^{*i}\;\;;\;\;\theta _{\bar{k}j}^{*i}\bar{\eta}^k=\theta
_j^{*i}\;\;;\;\;\theta _{\bar{k}\bar{h}}^{*i}\bar{\eta}^k=0;\;  \label{2'} \\
\theta _{kjr}^{*i}\eta ^k &=&-\theta _{jr}^{*i}\;\;;\;\;\theta _{k\bar{h}%
j}^{*i}\eta ^k=0\;\;;\;\;\theta _{r\bar{k}j}^{*i}\bar{\eta}^k=\theta
_{rj}^{*i}\;\;;\;\;\theta _{j\bar{k}\bar{h}}^{*i}\bar{\eta}^k=0;\;  \nonumber
\\
\theta _{k\bar{h}\bar{r}}^{*i}\eta ^k &=&\theta _{\bar{h}\bar{r}%
}^{*i}\;\;;\;\;(\dot{\partial}_h\theta _{kjr}^{*i})\eta ^k=-2\theta
_{hjr}^{*i}\;\;;\;\;(\dot{\partial}_r\theta _{k\bar{h}j}^{*i})\eta
^k=-\theta _{r\bar{h}j}^{*i}\;;\;  \nonumber \\
(\dot{\partial}_{\bar{h}}\theta _{r\bar{k}j}^{*i})\bar{\eta}^k &=&0\;\;;\;\;(%
\dot{\partial}_h\theta _{r\bar{k}j}^{*i})\bar{\eta}^k=\theta
_{rjh}^{*i}\;\;;\;\;(\dot{\partial}_{\bar{r}}\theta _{k\bar{h}j}^{*i})\eta
^k=0.  \nonumber
\end{eqnarray}

\begin{proposition}
Let $(M,F)$ be a complex Finsler space. If $D_{j\bar{k}h}^i=0$ then $F$ is
generalized Berwald and
\begin{eqnarray}
D_{jkh}^i &=&-\frac 12\{\theta _{jkh}^{*i}-\frac 1n[(\dot{\partial}_h\theta
_{ljk}^{*l})\eta ^i+\sum_{(j,k,h)}\theta _{ljh}^{*l}\delta _k^i]\};
\label{1'} \\
D_{j\bar{k}\bar{h}}^i &=&-\frac 12\{\theta _{j\bar{k}\bar{h}}^{*i}-\frac 1n[(%
\dot{\partial}_{\bar{h}}\theta _{l\bar{k}j}^{*l})\eta ^i+\theta _{l\bar{k}%
\bar{h}}^{*l}\delta _j^i]\};  \nonumber \\
\theta _{j\bar{k}h}^{*i} &=&\frac 1n[(\dot{\partial}_h\theta _{lj\bar{k}%
}^{*l})\eta ^i+\theta _{l\bar{k}j}^{*l}\delta _h^i+\theta _{l\bar{k}%
h}^{*l}\delta _j^i].  \nonumber
\end{eqnarray}
\end{proposition}

\begin{proof}
If $D_{j\bar{k}h}^i=0$ then
\begin{eqnarray}
G_{j\bar{k}h}^i &=&\frac 1{n+1}[(\dot{\partial}_hD_{\bar{k}j})\eta ^i+D_{%
\bar{k}j}\delta _h^i+D_{\bar{k}h}\delta _j^i]  \label{1} \\
&&+\frac 12\{\theta _{j\bar{k}h}^{*i}-\frac 1n[(\dot{\partial}_h\theta _{l%
\bar{k}j}^{*l})\eta ^i+\theta _{l\bar{k}j}^{*l}\delta _h^i+\theta _{l\bar{k}%
h}^{*l}\delta _j^i]\}  \nonumber
\end{eqnarray}
which will be contracted by $\eta ^j\eta ^h$ and then by $\eta _i$.

Using $G_{j\bar{k}h}^i\eta ^j\eta ^h=G_{h\bar{k}}^i$ $\eta ^h=\dot{\partial}%
_{\bar{k}}G^i$ ; $(\dot{\partial}_jD_{\bar{k}h})\eta ^j\eta ^h=0$ ; $D_{\bar{%
k}h}\eta ^h=G_{l\bar{k}}^l$ and taking into account (\ref{2'}), after the
contraction by $\eta ^j\eta ^h$ of $G_{j\bar{k}h}^i$, we obtain
\[
\dot{\partial}_{\bar{k}}G^i=\frac 2{n+1}G_{l\bar{k}}^l\eta ^i.
\]
Due to Lemma 2.1, i.e. $(\dot{\partial}_{\bar{k}}G^i)\eta _i=0,$ the
contraction of the above relation with $\eta _i$ leads to $G_{l\bar{k}}^l=0.$
Its differential with respect to $\eta ^h$ gives $G_{l\bar{k}h}^l=0$, i.e. $%
D_{\bar{k}h}=0$ which plugged into (\ref{1}) yields
\[
G_{j\bar{k}h}^i=\frac 12\{\theta _{j\bar{k}h}^{*i}-\frac 1n[(\dot{\partial}%
_h\theta _{l\bar{k}j}^{*l})\eta ^i+\theta _{l\bar{k}j}^{*l}\delta
_h^i+\theta _{l\bar{k}h}^{*l}\delta _j^i]\}.
\]
The last relation contracted by $\eta ^j$ gives $G_{h\bar{k}}^i=0.$
Next, it results $\dot{\partial}_{\bar{k}}G_{jh}^i=0$ which means that $%
G_{jh}^i$ are holomorphic functions with respect to $\eta $. Together with
their $0$ - homogeneity imply $G_{jh}^i=G_{jh}^i(z).$ Hence $G_{jkh}^i=G_{j%
\bar{k}\bar{h}}^i=0$ and (\ref{1'}).
\end{proof}

\begin{proposition}
Let $(M,F)$ be a complex Finsler space. If $D_{j\bar{k}\bar{h}}^i=0$ then $F$
is generalized Berwald and
\begin{eqnarray}
D_{jkh}^i &=&-\frac 12\{\theta _{jkh}^{*i}-\frac 1n[(\dot{\partial}_h\theta
_{ljk}^{*l})\eta ^i+\sum_{(j,k,h)}\theta _{ljh}^{*l}\delta _k^i]\};
\label{30} \\
\theta _{j\bar{k}\bar{h}}^{*i} &=&\frac 1n[(\dot{\partial}_{\bar{h}}\theta
_{lj\bar{k}}^{*l})\eta ^i+\theta _{l\bar{k}\bar{h}}^{*l}\delta _j^i];
\nonumber \\
D_{j\bar{k}h}^i &=&-\frac 12\{\theta _{j\bar{k}h}^{*i}-\frac 1n[(\dot{%
\partial}_h\theta _{lj\bar{k}}^{*l})\eta ^i+\theta _{lj\bar{k}}^{*l}\delta
_h^i+\theta _{l\bar{k}h}^{*l}\delta _j^i]\}.  \nonumber
\end{eqnarray}
\end{proposition}

\begin{proof}
If $D_{j\bar{k}\bar{h}}^i=0$ then
\begin{equation}
G_{j\bar{k}\bar{h}}^i=\frac 1{n+1}[(\dot{\partial}_jD_{\bar{k}\bar{h}})\eta
^i+D_{\bar{k}\bar{h}}\delta _j^i]+\frac 12\{\theta _{j\bar{k}\bar{h}%
}^{*i}-\frac 1n[(\dot{\partial}_{\bar{h}}\theta _{lj\bar{k}}^{*l})\eta
^i+\theta _{l\bar{k}\bar{h}}^{*l}\delta _j^i]\}.  \label{2''}
\end{equation}
The contraction of (\ref{2''}) by $\eta ^j\bar{\eta}^h\eta _i$ and using (%
\ref{2'}) and $G_{j\bar{k}\bar{h}}^i\bar{\eta}^h\eta ^j=-G_{j\bar{k}}^i\eta
^j=-\dot{\partial}_{\bar{k}}G^i$ ; $D_{\bar{k}\bar{h}}\bar{\eta}^h=-G_{i\bar{%
k}}^i$ ; $(\dot{\partial}_jD_{\bar{k}\bar{h}})\bar{\eta}^h\eta ^j=-(\dot{%
\partial}_jG_{i\bar{k}}^i)\eta ^j=-G_{i\bar{k}}^i,$ it results
\[
0=-(\dot{\partial}_{\bar{k}}G^i)\eta _i=-\frac{2L}{n+1}G_{l\bar{k}}^l,
\]
which implies $G_{l\bar{k}}^l=0$ and so, $G_{l\bar{k}\bar{h}}^l=0$, i.e. $D_{%
\bar{k}\bar{h}}=0.$ Plugging $D_{\bar{k}\bar{h}}=0$ into (\ref{2''}) we
obtain
\[
G_{j\bar{k}\bar{h}}^i=\frac 12\{\theta _{j\bar{k}\bar{h}}^{*i}-\frac 1n[(%
\dot{\partial}_{\bar{h}}\theta _{lj\bar{k}}^{*l})\eta ^i+\theta _{l\bar{k}%
\bar{h}}^{*l}\delta _j^i]\}.
\]
Now, the last relations contracted only by $\bar{\eta}^h$ leads to $G_{j\bar{%
k}}^i=0.$ As above we obtain that $G_{jh}^i$ depend only on $z.$ So, the
space is generalized Berwald and the relations (\ref{30}) are true.
\end{proof}

\begin{theorem}
Let $(M,F)$ be a complex Finsler space. Then, $(M,F)$ is Douglas if and only
if it is generalized Berwald with
\begin{eqnarray}
\theta _{jkh}^{*i}=\frac 1n[(\dot{\partial}_h\theta _{ljk}^{*l})\eta
^i+\sum_{(j,k,h)}\theta _{ljh}^{*l}\delta _k^i];  \label{4} \\
\theta _{j\bar{k}\bar{h}}^{*i}=\frac 1n[(\dot{\partial}_{\bar{h}}\theta _{lj%
\bar{k}}^{*l})\eta ^i+\theta _{l\bar{k}\bar{h}}^{*l}\delta _j^i];  \nonumber
\\
\theta _{j\bar{k}h}^{*i}=\frac 1n[(\dot{\partial}_h\theta _{lj\bar{k}%
}^{*l})\eta ^i+\theta _{lj\bar{k}}^{*l}\delta _h^i+\theta _{l\bar{k}%
h}^{*l}\delta _j^i]\}.  \nonumber
\end{eqnarray}
\end{theorem}

\begin{proof}
The direct implication is obvious by the last two
Propositions. Conversely, if the space is generalized Berwald, replacing the
relations (\ref{4}) into (\ref{17}), it follows $D_{j{\bar{k%
}}h}^i=D_{jkh}^i=D_{j{\bar{k}}{%
\bar{h}}}^i=0$.
\end{proof}

\section{Weakly K\"{a}hler projective changes}

\setlength\arraycolsep{3pt}\setcounter{equation}{0} All the next discussion
will be focused on the weakly K\"{a}hler complex Finsler spaces. In this
case, the projective invariants of Douglas type (\ref{17}) are
\begin{eqnarray}
D_{jkh}^i &=&G_{jkh}^i-\frac 1{n+1}[(\dot{\partial}_jD_{kh})\eta
^i+\sum_{(j,k,h)}D_{jk}\delta _h^i];  \label{9} \\
D_{j\bar{k}\bar{h}}^i &=&G_{j\bar{k}\bar{h}}^i-\frac 1{n+1}[(\dot{\partial}%
_jD_{\bar{k}\bar{h}})\eta ^i+D_{\bar{k}\bar{h}}\delta _j^i];  \nonumber \\
D_{j\bar{k}h}^i &=&G_{j\bar{k}h}^i-\frac 1{n+1}[(\dot{\partial}_jD_{\bar{k}%
h})\eta ^i+D_{\bar{k}h}\delta _j^i+D_{\bar{k}j}\delta _h^i].  \nonumber
\end{eqnarray}

By Lemma 4.1, it immediately results the following Proposition.

\begin{proposition}
Let $(M,F)$ be a weakly K\"{a}hler complex Finsler space. If one of $hv$-, $%
\bar{h}\bar{v}$- or $h\bar{v}$- Ricci tensors is vanishing, then
\begin{equation}
D_{jkh}^i=G_{jkh}^i\;;\;D_{j\bar{k}\bar{h}}^i=G_{j\bar{k}\bar{h}}^i\;;\;D_{j%
\bar{k}h}^i=G_{j\bar{k}h}^i.  \label{10}
\end{equation}
\end{proposition}

\begin{proposition}
Let $(M,F)$ be a weakly K\"{a}hler complex Finsler space. If one of
statements (\ref{10}) is true, then the $hv$-, $\bar{h}\bar{v}$- and $h\bar{v%
}$- Ricci tensors are vanishing.
\end{proposition}

\begin{proof}
Suppose that $D_{j\bar{k}\bar{h}}^i=G_{j\bar{k}\bar{h}}^i$.
Then, using (\ref{9}) it results $(\dot{\partial}_jD_{\bar{k}\bar{h}})\eta
^i+D_{\bar{k}\bar{h}}\delta _j^i=0.$ Since $(\dot{\partial}_jD_{\bar{k}\bar{h%
}})\eta ^j=D_{\bar{k}\bar{h}},$ hence $(n+1)D_{\bar{k}\bar{h}}=0,$ and so $%
D_{\bar{k}\bar{h}}=0.$ By Lemma 4.1, $hv$-, and $h\bar{v}$- Ricci tensors
are vanishing. The proof is similar for $D_{jkh}^i=G_{jkh}^i$ or $D_{j\bar{%
k}\bar{h}}^i=G_{j\bar{k}\bar{h}}^i.$
\end{proof}

Corroborating (\ref{9}) with Propositions 4.2 and 4.3, it follows

\begin{corollary}
Let $(M,F)$ be a weakly K\"{a}hler complex Finsler space.

i) If $D_{j\bar{k}h}^i=0$ then $D_{jkh}^i=D_{j\bar{k}\bar{h}}^i=0.$

ii) If $D_{j\bar{k}\bar{h}}^i=0$ then $D_{jkh}^i=D_{j\bar{k}h}^i=0.$
\end{corollary}

\begin{theorem}
Let $(M,F)$ be a weakly K\"{a}hler complex Finsler space. If either $D_{j%
\bar{k}\bar{h}}^i=0$ or $D_{j\bar{k}h}^i=0$ then the space is complex
Berwald.
\end{theorem}

\begin{proof}
If either $D_{j\bar{k}\bar{h}}^i=0$ or $D_{j\bar{k}h}^i=0$
then $G_{jh}^i=G_{jh}^i(z)$, which means that the space is generalized
Berwald. The proof is completed by Theorem 2.1.
\end{proof}

\begin{theorem}
If $(M,F)$ is a complex weakly K\"{a}hler Douglas space then it is Berwald.
\end{theorem}

\begin{proof}
It results by Theorem 5.1.
\end{proof}

Note that, the weakly K\"{a}hler property is preserved by the projective
changes (for proof details see Theorem 3.2 from \cite{Al-Mu}), and then we
have
\begin{equation}
\tilde{G}^i=G^i+P\eta ^i,  \label{9''}
\end{equation}
where $P$ is a $(1,0)$ - homogeneous function. Under this projective change,
we obtain
\begin{eqnarray}
\stackrel{c}{\tilde{N}_j^i} &=&\stackrel{c}{N_j^i}+P_j\eta ^i+P\delta
_j^i\;;\;\stackrel{c}{\tilde{\delta}_k}=\stackrel{c}{\delta _k}-(P_k\eta
^i+P\delta _k^i)\dot{\partial}_i;  \label{9'} \\
\tilde{G}_{jk}^i &=&G_{jk}^i+P_{jk}\eta ^i+P_k\delta _j^i+P_j\delta _k^i\;;\;%
\tilde{G}_{j\bar{k}}^i=G_{j\bar{k}}^i+P_{j\bar{k}}\eta ^i+P_{\bar{k}}\delta
_j^i,  \nonumber
\end{eqnarray}
where $P_{jk}:=\dot{\partial}_kP_j=P_{kj},$ $P_{\bar{k}}:=\dot{\partial}_{%
\bar{k}}P,$ $P_{j\bar{k}}:=\dot{\partial}_{\bar{k}}P_j=\dot{\partial}_jP_{%
\bar{k}}.$ Moreover, the $(1,0)$ - homogeneity of $P$ implies
\begin{equation}
P_k\eta ^k=P\;;\;P_{j\bar{k}}\bar{\eta}^k=0\;;\;P_{jk}\eta ^k=0\;;\;P_{\bar{k%
}}\bar{\eta}^k=0\;;\;P_{j\bar{k}}\eta ^j=P_{\bar{k}}.  \nonumber
\end{equation}

Next, we shall study the $hh-$ curvatures tensor $K_{jkh}^i$. Under the
projective change (\ref{9''}), we have
\begin{eqnarray}
\tilde{K}_{kh}^i &=&K_{kh}^i+\mathcal{A}_{(k,h)}[P_{k\stackrel{B}{|}h}\eta
^i+(P_{\stackrel{B}{|}h}-PP_h)\delta _k^i];  \label{20} \\
\tilde{K}_{jkh}^i &=&K_{jkh}^i+\mathcal{A}_{(k,h)}[P_{jk\stackrel{B}{|}%
h}\eta ^i+P_{k\stackrel{B}{|}h}\delta _j^i+(P_{j\stackrel{B}{|}%
h}-P_jP_h-PP_{jh})\delta _k^i],  \nonumber
\end{eqnarray}
where '$_{\stackrel{B}{|}h}$' is the horizontal covariant derivative with
respect to $B\Gamma $ and $\mathcal{A}_{(k,h)}$ is the alternate operator,
for example $\mathcal{A}_{(k,h)}\{P_{k\stackrel{B}{|}h}\}:=P_{k\stackrel{B}{|%
}h}-P_{h\stackrel{B}{|}k}.$ Next we make the following notations
\[
X_{kh}:=P_{k\stackrel{B}{|}h}-P_{h\stackrel{B}{|}k}\;;\;X_h:=P_{\stackrel{B}{%
|}h}-PP_h
\]
which have the properties
\begin{eqnarray*}
\dot{\partial}_jX_h &=&P_{j\stackrel{B}{|}h}-P_jP_h-PP_{jh}\;;\;\dot{\partial%
}_jX_h-\dot{\partial}_hX_j=P_{j\stackrel{B}{|}h}-P_{h\stackrel{B}{|}%
j}\;=X_{jh};\; \\
\dot{\partial}_jX_{kh} &=&P_{kj\stackrel{B}{|}h}-P_{hj\stackrel{B}{|}k}\;;\;(%
\dot{\partial}_jX_h)\eta ^j=X_h\;;\;(\dot{\partial}_jX_{kh})\eta ^j=0\;;\; \\
X_{kj}\eta ^j &=&P_{k\stackrel{B}{|}0}-P_{\stackrel{B}{|}k}:=X_{k0}.
\end{eqnarray*}

By means of these, the changes (\ref{20}) become
\begin{eqnarray}
\tilde{K}_{kh}^i &=&K_{kh}^i+X_{kh}\eta ^i+X_h\delta _k^i-X_k\delta _h^i
\nonumber \\
\tilde{K}_{jkh}^i &=&K_{jkh}^i+(\dot{\partial}_jX_{kh})\eta ^i+X_{kh}\delta
_j^i+(\dot{\partial}_jX_h)\delta _k^i-(\dot{\partial}_jX_k)\delta _h^i.
\label{21}
\end{eqnarray}

Now, we introduce the $hh$- Ricci tensor $K_{kh}:=K_{ikh}^i$. Another
important tensor is $H_{jk}:=K_{jki}^i$. The link between these horizontal
curvature tensors is $H_{kj}-H_{jk}=K_{jk}.$ Summing by $i=j$ and then $i=h$
together with a contraction by $\eta ^j,$ in the second relation from (\ref
{21}), it yields
\begin{eqnarray}
X_{kh} &=&\frac 1{n+1}(\tilde{K}_{kh}-K_{kh})=\frac 1{n+1}[(\tilde{H}_{hk}-%
\tilde{H}_{kh})-(H_{hk}-H_{kh})];\;  \label{22} \\
\tilde{H}_{0k} &=&H_{0k}+X_{k0}-(n-1)X_k\;.  \nonumber
\end{eqnarray}
From here, it results
\begin{eqnarray}
X_{k0} &=&\frac 1{n+1}[(\tilde{H}_{0k}-\tilde{H}_{k0})-(H_{0k}-H_{k0})];
\label{23} \\
X_k &=&-\frac 1{n+1}(\tilde{H}_k-H_k)\;\mbox{with}\;H_k:=\frac
1{n-1}(nH_{0k}+H_{k0}),  \nonumber
\end{eqnarray}
for any $n\geq 2.$ Moreover,
\begin{equation}
K_{jk}=\dot{\partial}_jH_{k0}-\dot{\partial}_kH_{j0}=\dot{\partial}_kH_{0j}-%
\dot{\partial}_jH_{0k}\;\mbox{and}\;H_{jk}=\dot{\partial}_jH_{0k}.
\label{24}
\end{equation}
Now, substituting (\ref{22}) and (\ref{23}) in (\ref{21}) we obtain the
following invariants
\begin{eqnarray}
W_{kh}^i &=&K_{kh}^i+\frac 1{n+1}\mathcal{A}_{(k,h)}(H_{kh}\eta ^i+H_h\delta
_k^i);  \label{25} \\
W_{jkh}^i &=&K_{jkh}^i+\frac 1{n+1}\mathcal{A}_{(k,h)}[(\dot{\partial}%
_jH_{kh})\eta ^i+H_{kh}\delta _j^i+(\dot{\partial}_jH_h)\delta _k^i],
\nonumber
\end{eqnarray}
in which the second formula is a \textit{\ projective curvature invariant of
Weyl type}. Note that, if $(M,F)$ is K\"{a}hler, then $W_{jkh}^i=0.$

\begin{theorem}
Let $(M,F)$ be a weakly K\"{a}hler complex Finsler space of complex
dimension $n\geq 2.$

i) Then, $W_{jkh}^i=0$ if and only if $W_{kh}^i=0;$

ii) If $K_{kh}=0$ then $W_{jkh}^i=K_{jkh}^i+\frac 1{n-1}(H_{jh}\delta
_k^i-H_{jk}\delta _h^i);$

iii) If $H_{kh}=0$ then $W_{jkh}^i=K_{jkh}^i.$
\end{theorem}

\begin{proof}
i) If $W_{jkh}^i=0,$ then

$K_{jkh}^i=-\frac 1{n+1}\mathcal{A}_{(k,h)}[(\dot{\partial}_jH_{kh})\eta
^i+H_{kh}\delta _j^i+(\dot{\partial}_jH_h)\delta _k^i],$\newline
which contracted by $\eta ^j$ give $K_{kh}^i=-\frac 1{n+1}\mathcal{A}%
_{(k,h)}(H_{kh}\eta ^i+H_h\delta _k^i)$ and hence, $W_{kh}^i=0.$

Conversely, if $W_{kh}^i=0$ then $K_{kh}^i=-\frac 1{n+1}\mathcal{A}%
_{(k,h)}(H_{kh}\eta ^i+H_h\delta _k^i)$. Differentiating with respect to $\eta ^j$, it results

$K_{jkh}^i=-\frac 1{n+1}\mathcal{A}_{(k,h)}[(\dot{\partial}_jH_{kh})\eta
^i+H_{kh}\delta _j^i+(\dot{\partial}_jH_h)\delta _k^i]$, that is, $%
W_{jkh}^i=0. $

ii) If $K_{kh}=0$ then $H_{kj}=H_{jk}.$ Substituting into (\ref{25}) and
using (\ref{23}) and (\ref{24}), it results our claim. iii) immediately results by (\ref{25}%
) and (\ref{23}).
\end{proof}

In order to obtain another projective curvature invariant of Weyl type we
assume that the weakly K\"{a}hler complex Finsler metric $F$ is generalized
Berwald. Thus, we have $K_{j\bar{k}\bar{h}}^i=0$, $K_{j\bar{k}h}^i=-%
\stackrel{c}{\delta _{\bar{k}}}G_{jh}^i$ and the Bianchi identities get $%
\dot{\partial}_rK_{j\overline{k}h}^i=0$ and $\dot{\partial}_{\bar{r}}K_{j%
\overline{k}h}^i=0.$

Note that by a projective change, the generalized Berwald property of the
metric $L$ is transferred to the metric $\tilde{L}.$ Moreover, the
generalized Berwald property together with the weakly K\"{a}hler assumption
implies that $F$ and $\tilde{F}$ are complex Berwald metrics (Theorem 2.1).
Hence, $K_{j\bar{k}h}^i=-\delta _{\bar{k}}L_{jh}^i.$ Therefore, under these
assumptions, the function $P$ from the projective change (\ref{9''}) is
holomorphic with respect to $\eta $, i.e. $P_{\bar{k}}=0,$ (see Proposition
3.1 from \cite{Al-Mu}), and
\begin{eqnarray}
\tilde{N}_j^i &=&N_j^i+P_j\eta ^i+P\delta _j^i\;;\;\tilde{\delta}_k=\delta
_k-(P_k\eta ^i+P\delta _k^i)\dot{\partial}_i;  \label{9'''} \\
\tilde{L}_{jk}^i &=&L_{jk}^i+P_{jk}\eta ^i+P_k\delta _j^i+P_j\delta _k^i\;;\;%
\tilde{G}_{j\bar{k}}^i=G_{j\bar{k}}^i=0.  \nonumber
\end{eqnarray}
Consequently,
\begin{eqnarray}
\tilde{K}_{j\bar{k}h}^i &=&K_{j\bar{k}h}^i-P_{jh|\bar{k}}\eta ^i-P_{j|\bar{k}%
}\delta _h^i-P_{h|\bar{k}}\delta _j^i;  \label{10'} \\
0 &=&P_{jhr|\bar{k}}\eta ^i+P_{jh|\bar{k}}\delta _r^i+P_{jr|\bar{k}}\delta
_h^i+P_{hr|\bar{k}}\delta _j^i.  \nonumber
\end{eqnarray}

Next, we consider the $h\bar{h}$ - Ricci tensor $K_{\bar{k}h}:=K_{i\bar{k}%
h}^i.$ Since $F$ is K\"{a}hler, $K_{i\bar{k}h}^i=K_{h\bar{k}i}^i.$ Making $%
i=j$ in (\ref{10'}), it results
\begin{eqnarray}
P_{h|\bar{k}} &=&-\frac 1{n+1}(\tilde{K}_{\bar{k}h}-K_{\bar{k}h});
\label{11} \\
P_{hr|\bar{k}} &=&0,  \nonumber
\end{eqnarray}
which substituted into the first equation from (\ref{10'}), give a new
\textit{projective curvature invariant of Weyl type,} which is valid only
for the complex Berwald spaces,
\begin{equation}
W_{j\bar{k}h}^i=K_{j\bar{k}h}^i-\frac 1{n+1}(K_{\bar{k}j}\delta _h^i+K_{\bar{%
k}h}\delta _j^i).  \label{120}
\end{equation}

Note that for any complex Berwald space, the $h\bar{h}$ - curvatures
coefficients of Chern-Finsler connection can be rewritten as $R_{j\bar{k}%
h}^i=K_{j\bar{k}h}^i+K_{m\bar{k}h}^l\eta ^mC_{jl}^i.$ So that, $R_{\bar{r}j%
\bar{k}h}=K_{\bar{r}j\bar{k}h}+K_{m\bar{k}h}^l\eta ^mC_{j\bar{r}l}$, where $%
K_{\bar{r}j\bar{k}h}:=K_{j\bar{k}h}^ig_{i\bar{r}},$ and $R_{\bar{r}j\bar{k}%
h}\eta ^j=K_{\bar{r}j\bar{k}h}\eta ^j.$ This implies
\[
\mathcal{K}_F(z,\eta )=\frac 2{L^2}K_{\bar{r}j\bar{k}h}\bar{\eta}^r\eta ^j%
\bar{\eta}^k\eta ^h.
\]

\begin{theorem}
Let $(M,F)$ be a connected complex Berwald space of complex dimension $n\geq
2.$ Then, $W_{j\bar{k}h}^i=0$ if and only if $K_{\bar{m}j\bar{k}h}=\frac{%
\mathcal{K}_F}4(g_{j\bar{k}}g_{h\bar{m}}+g_{h\bar{k}}g_{j\bar{m}}).$ In this
case, $\mathcal{K}_F=c,$ where $c$ is a constant on $M$ and the space is
either purely Hermitian with $K_{\bar{k}j}=\frac{c(n+1)}4g_{j\bar{k}}$ or
non purely Hermitian with $c=0$ and $K_{j\bar{k}h}^i=0.$
\end{theorem}

\begin{proof}
Using (\ref{120}) and $W_{j\bar{k}h}^i=0$, it results
\begin{equation}
K_{j\bar{k}h}^i=\frac 1{n+1}(K_{\bar{k}j}\delta _h^i+K_{\bar{k}h}\delta _j^i)
\label{2}
\end{equation}
which contracted with $g_{i\bar{m}}$ gives
\begin{equation}
K_{\bar{m}j\bar{k}h}=\frac 1{n+1}(K_{\bar{k}j}g_{h\bar{m}}+K_{\bar{k}h}g_{j%
\bar{m}}),  \label{3}
\end{equation}
and
\begin{equation}
R_{\bar{m}j\bar{k}h}=\frac 1{n+1}(K_{\bar{k}j}g_{h\bar{m}}+K_{\bar{k}h}g_{j%
\bar{m}}+K_{\bar{k}l}\eta ^lC_{j\bar{m}h}),  \label{3'}
\end{equation}
where $C_{j\bar{m}h}:=\dot{\partial}_hg_{j\bar{m}}$.

Since $R_{\bar{r}j\bar{k}h}=R_{\bar{r}h\bar{k}j}$, see \cite{A-P} p. 105, it
results $R_{\bar{r}j\bar{k}h}=R_{\bar{k}j\bar{r}h},$ and therefore,
\begin{equation}
K_{\bar{r}j\bar{k}h}\eta ^j=K_{\bar{k}j\bar{r}h}\eta ^j.  \label{100}
\end{equation}
From (\ref{3}) also results
\begin{equation}
\mathcal{K}_F=\frac 4{L(n+1)}K_{\bar{k}j}\eta ^j\bar{\eta}^k,  \label{50}
\end{equation}
which, indeed, can be rewritten as $L\mathcal{K}_F=\frac 4{n+1}K_{\bar{k}%
j}\eta ^j\bar{\eta}^k.$ Differentiating this last formula with respect to $%
\bar{\eta}^m$ and using again the Bianchi identity $\dot{\partial}_{\bar{m}%
}K_{\overline{k}h}=0,$ it follows that $\mathcal{K}_F\bar{\eta}_m+L(\dot{%
\partial}_{\bar{m}}\mathcal{K}_F)=\frac 4{n+1}K_{\bar{m}j}\eta ^j$.

Now, due to (\ref{100}), we obtain
\begin{equation}
\mathcal{K}_F\bar{\eta}_m=\frac 4{n+1}K_{\bar{m}j}\eta ^j.  \label{51}
\end{equation}
Thus, $L(\dot{\partial}_{\bar{m}}\mathcal{K}_F)=0$ and so, $\mathcal{K}_F$
depends only on $z.$ Differentiating (\ref{51}) with respect to $\eta ^l,$
it gives $K_{\bar{m}l}=\frac{(n+1)\mathcal{K}_F}4g_{l\bar{m}},$ which
plugged into (\ref{3}) yields $K_{\bar{m}j\bar{k}h}=\frac{\mathcal{K}_F}%
4(g_{j\bar{k}}g_{h\bar{m}}+g_{h\bar{k}}g_{j\bar{m}}).$

Conversely, since $K_{j\bar{k}h}^i=\frac{\mathcal{K}_F}4(g_{j\bar{k}}\delta
_h^i+g_{h\bar{k}}\delta _j^i)$ and $K_{\bar{k}h}=\frac{(n+1)\mathcal{K}_F}%
4g_{h\bar{k}},$ the relation (\ref{120}) implies $W_{j\bar{k}h}^i=0.$

In order to prove that $\mathcal{K}_F$ is a constant on $M$ we use the
Bianchi identity $K_{j\bar{r}k|\bar{h}}^i=K_{j\bar{h}k|\bar{r}}^i$ from (\ref
{111}). Contracting by $g_{i\bar{m}}\bar{\eta}^m\eta ^j\bar{\eta}^r\eta ^k,$
it gives
\begin{equation}
\mathcal{K}_{F|\overline{h}}=\frac 1L\mathcal{K}_{F|\bar{0}}\overline{\eta }%
_h.  \label{IV.3}
\end{equation}

Taking into account $\mathcal{K}_{F|\overline{h}}|_j=\mathcal{K}_F|_{j|%
\overline{h}}=0,$ where '$|_k$ ' is the vertical covariant derivative with
respect to Chern-Finsler connection, and deriving (\ref{IV.3}), we easily
deduce
\[
0=\mathcal{K}_{F|\overline{h}}|_j=\frac 1L\mathcal{K}_{F|\bar{0}}(g_{j%
\overline{h}}-\frac 1L\eta _j\overline{\eta }_h),
\]
which multiplied by $g^{\bar{h}j}$, it gets $\frac 1L(n-1)\mathcal{K}_{F|%
\bar{0}}=0.$ Plugging it into (\ref{IV.3}), it follows that $\mathcal{K}_{F|%
\overline{h}}=0,$ i.e. $\frac{\partial \mathcal{K}_F}{\partial \overline{z}^h%
}=0.$ By conjugation, $\frac{\partial \mathcal{K}_F}{\partial z^h}=0$ and
so, $\mathcal{K}_F$ is a constant $c$ on $M$. This implies $K_{\bar{k}j}=%
\frac{c(n+1)}4g_{j\bar{k}}$ and its derivative with respect to $\eta ^l$
leads to $c\;\dot{\partial}_lg_{j\bar{k}}=0$, and hence the last claim.
\end{proof}

\section{Locally projectively flat complex Finsler metrics}

\setlength\arraycolsep{3pt} \setcounter{equation}{0}Using some ideas from
the real case, we shall define the locally projectively flat complex Finsler
metrics.

Let $\tilde{L}$ be a locally Minkowski complex Finsler metric on the
underlying manifold $M.$ Corresponding to the metric $\tilde{L},$ at any
point of $M$ there exist local charts in which the fundamental metric tensor
$\tilde{g}_{i{\bar{j}}}$ depends only on $\eta $ and thus, the spray
coefficients $\tilde{G}^i=0$ and the functions $\tilde{\theta}^{*i}=0$, in
such local charts. The complex Finsler metrics $L$ will be called \textit{%
locally projectively flat} if it is projectively related to the locally
Minkowski metric $\tilde{L}.$ Since the weakly K\"{a}hler property is
preserved under the projective change, any locally projectively flat metric
is weakly K\"{a}hler. We recall Theorem 3.3 from \cite{Al-Mu},

\begin{theorem}
Let $L$ and $\tilde{L}$ be complex Finsler metrics on the manifold $M$.
Then, $L$ and $\tilde{L}$ are projectively related if and only if
\begin{equation}
\frac 12[\dot{\partial}_{\bar{r}}(\delta _k\tilde{L})\eta ^k+2(\dot{\partial}%
_{\bar{r}}G^l)(\dot{\partial}_l\tilde{L})]=P(\dot{\partial}_{\bar{r}}\tilde{L%
})+B^i\tilde{g}_{i\bar{r}}\;;\;r=\overline{1,n},  \label{3.3}
\end{equation}
with $\displaystyle P=\frac 1{2\tilde{L}}[(\delta _k\tilde{L})\eta ^k+\theta
^{*i}(\dot{\partial}_i\tilde{L})]$ and $B^i:=\frac 12(\tilde{\theta}%
^{*i}-\theta ^{*i}).$
\end{theorem}

\begin{theorem}
$L$ is locally projectively flat if and only if it is weakly K\"{a}hler and
\begin{equation}
\dot{\partial}_{\bar{r}}(\delta _k\tilde{L})\eta ^k+2(\dot{\partial}_{\bar{r}%
}G^l)(\dot{\partial}_l\tilde{L})=2P(\dot{\partial}_{\bar{r}}\tilde{L})\;;\;r=%
\overline{1,n},  \label{3.2}
\end{equation}
where $P=\frac 1{2\tilde{L}}(\delta _k\tilde{L})\eta ^k.$ Moreover, $%
G^i=-P\eta ^i$.
\end{theorem}

\begin{proof}
The above equivalence results by Theorem 6.1 in which\textbf{%
\ }$\tilde{L}$ is a locally Minkowski metric on $M$. Taking into account $%
(\delta _k\tilde{L})\eta ^k=-2G^l(\dot{\partial}_l\tilde{L}),$ the condition
(\ref{3.2}) is equivalent to $-G^l\tilde{g}_{l\bar{r}}=P(\dot{\partial}_{%
\bar{r}}\tilde{L})$. By contraction with $\tilde{g}^{\bar{r}i}$, we obtain $%
G^i=-P\eta ^i.$
\end{proof}

\begin{proposition}
If $L$ is locally projectively flat then $G^i=\frac 1{2L}\frac{\partial L}{%
\partial z^k}\eta ^k\eta ^i.$
\end{proposition}

\begin{proof}
Since $G^i=\frac 12g^{\bar{m}i}\frac{\partial g_{r\bar{m}}}{\partial z^k}%
\eta ^k\eta ^r$ and $L$ is locally projectively flat, then $\frac 12g^{\bar{m%
}i}\frac{\partial g_{r\bar{m}}}{\partial z^k}\eta ^k\eta ^r=-P\eta ^i.$
Contracting by $\eta _i,$ it leads to $P=-\frac 1{2L}\frac{\partial L}{%
\partial z^k}\eta ^k$ which finishes the proof.
\end{proof}

\begin{proposition}
Let $(M,F)$ be a generalized Berwald space. If $L$ is locally projectively
flat then it is a complex Berwald metric with $W_{j\bar{k}h}^i=0.$
\end{proposition}

\begin{proof}
By Theorem 2.1, $L$ is a complex Berwald metric. Since $%
\tilde{K}_{j\bar{k}h}^i=\tilde{K}_{\bar{k}h}=0,$ the relations (\ref{10'})
and (\ref{11}), give $K_{j\bar{k}h}^i=\frac 1{n+1}(K_{\bar{k}j}\delta
_h^i+K_{\bar{k}h}\delta _j^i)$ and so, $W_{j\bar{k}h}^i=0.$
\end{proof}

By Theorem 5.4 we have proved

\begin{theorem}
Let $(M,F)$ be a connected generalized Berwald space of complex dimension $%
n\geq 2.$ If $L$ is locally projectively flat then it is of constant
holomorphic curvature. Moreover, if the constant value of the holomorphic
curvature is non-zero, then $(M,F)$ is a purely Hermitian space.
\end{theorem}

Next we study as an application the weakly K\"{a}hler complex Finsler
metrics $L$ with the spray coefficients $G^i=\rho _r\eta ^r\eta ^i,$ where $%
\rho $ is a smooth complex function depending only on $z\in M$, $\rho _r:=%
\frac{\partial \rho }{\partial z^r}$ and $\rho _{r\bar{h}}:=\frac{\partial
\rho _r}{\partial \bar{z}^h}$ is Hermitian, i.e. $\overline{\rho _{r\bar{h}}}%
=\rho _{h\bar{r}},$ and it is nondegenerated.

\begin{theorem}
Let $(M,F)$ be a weakly K\"{a}hler complex Finsler space with $G^i=\rho
_r\eta ^r\eta ^i.$ Then

i) $L$ is locally projectively flat;

ii) $L$ is a complex Berwald metric;

iii) $L$ is a purely Hermitian metric of non-zero constant holomorphic
curvature $\mathcal{K}_F=-\frac 4L\rho _{r\bar{h}}\eta ^r\bar{\eta}^h.$

iv) $\rho $ satisfies the system of partial differential equations
\begin{equation}
\rho _{r\bar{h}k}=\rho _r\rho _{k\bar{h}}+\rho _k\rho _{r\bar{h}},
\label{3.5}
\end{equation}
where $\rho _{r\bar{h}k}:=\frac{\partial \rho _{r\bar{h}}}{\partial z^k}=%
\frac{\partial \rho _{k\bar{h}}}{\partial z^r}=\frac{\partial \rho _{rk}}{%
\partial \bar{z}^h}$ and $\rho _{r\bar{h}k}=\rho _{k\bar{h}r}$
\end{theorem}

\begin{proof}
In order to prove i), we use Theorem 6.2. Let $\tilde{L}$ be
a locally Minkowski metric on $M.$ Since $L$ is weakly K\"{a}hler, we must
show only that the equation (\ref{3.2}) is satisfied. Indeed, we have $\dot{%
\partial}_{\bar{r}}G^l=0,$ $(\delta _k\tilde{L})\eta ^k=-2G^l(\dot{\partial}%
_l\tilde{L})=-2\rho _r\eta ^r\eta ^l(\dot{\partial}_l\tilde{L})=-2\tilde{L}%
\rho _r\eta ^r,$ and so $\dot{\partial}_{\bar{r}}(\delta _k\tilde{L})\eta
^k=-2(\dot{\partial}_{\bar{r}}\tilde{L})\rho _l\eta ^l,$ which implies the
equation (\ref{3.2}).

Since $\dot{\partial}_{\bar{r}%
}G^l=0$, $L$ is generalized Berwald. Thus, Theorem 2.1 yields ii).

iii) Theorem 6.3 together with i) and ii) show that $W_{j\bar{k}h}^i=0$ and $%
L$ is of constant holomorphic curvature. Since $L$ is a complex Berwald
metric, $\delta _{\bar{k}}=\stackrel{c}{\delta }_{\bar{k}}$ and $%
L_{jh}^i=G_{jh}^i.$ Hence $K_{j\bar{k}h}^i=-\delta _{\bar{k}}L_{jh}^i$, which
will be rewritten in terms of derivatives of $\rho .$ Indeed, two successive
differentiations of the equations $G^i=\rho _r\eta ^r\eta ^i$ lead to
\begin{equation}
L_{jk}^i=\rho _k\delta _j^i+\rho _j\delta _k^i\;.  \label{3.5'}
\end{equation}
Consequently,
\[
K_{j\bar{k}h}^i=-\rho _{j\bar{k}}\delta _h^i-\rho _{h\bar{k}}\delta _j^i
\]
which gives $K_{\bar{r}j\bar{k}h}=-\rho _{j\bar{k}}g_{h\bar{r}}-\rho _{h\bar{%
k}}g_{j\bar{r}}$ and so,
\begin{equation}
\mathcal{K}_F=-\frac 4L\rho _{r\bar{h}}\eta ^r\bar{\eta}^h.  \label{3.6}
\end{equation}

Since $\rho _{r\bar{h}}$ is nondegenerated, $\mathcal{K}_F\neq 0$ and by
Theorem 6.3 it results that $L$ is a purely Hermitian metric.

iv) To establish the system (\ref{3.5}) we use (\ref{3.6}). This implies
\begin{equation}
L=-\frac 4{\mathcal{K}_F}\rho _{r\bar{h}}\eta ^r\bar{\eta}^h=g_{r\bar{h}%
}\eta ^r\bar{\eta}^h,  \label{3.7}
\end{equation}
which gives
\begin{equation}
g_{r\bar{h}}=-\frac 4{\mathcal{K}_F}\rho _{r\bar{h}}\;\mbox{and}\;\delta
_kg_{r\bar{h}}=-\frac 4{\mathcal{K}_F}\rho _{r\bar{h}k}.  \label{3.8}
\end{equation}
Now, using (\ref{1.3}) and (\ref{3.5'}) it results
\begin{equation}
\delta _kg_{j\bar{m}}=\rho _kg_{j\bar{m}}+\rho _jg_{k\bar{m}}  \label{3.9}
\end{equation}
The substitution of (\ref{3.8}) into (\ref{3.9}) implies (\ref{3.5}).
Moreover, the K\"{a}hler property of $L$ gives $\rho _{r\bar{h}k}=\rho _{k%
\bar{h}r}.$
\end{proof}

\textbf{Acknowledgment:} The first author is supported by the Sectorial
Operational Program Human Resources Development (SOP HRD), financed from the
European Social Fund and by Romanian Government under the Project number
POSDRU/89/1.5/S/59323.

\begin{flushleft}
Transilvania Univ., Faculty of Mathematics and Informatics

Iuliu Maniu 50, Bra\c{s}ov 500091, ROMANIA

e-mail: nicoleta.aldea@lycos.com

e-mail: gh.munteanu@unitbv.ro
\end{flushleft}

\end{document}